\documentclass{amsart}

\usepackage{amssymb, amsmath, amsfonts}
\usepackage{color}
\usepackage{graphicx}
\usepackage{hyperref}
\usepackage{mathrsfs}
\usepackage{pb-diagram}
\usepackage{epstopdf}
\usepackage{amsthm,enumerate,amscd,latexsym,curves}
\usepackage{bbm}
\usepackage[mathscr]{eucal}
\usepackage{epsfig,epsf,xypic,epic}
\usepackage{caption}
\usepackage{subcaption}

\newtheorem{thm}{Theorem}[section]
\newtheorem{prop}[thm]{Proposition}
\newtheorem{lemma}[thm]{Lemma}

\newtheorem{defn}[thm]{Definition}
\newtheorem{example}[thm]{Example}

\newtheorem{remark}[thm]{Remark}

\numberwithin{equation}{section}

 \numberwithin{equation}{section}

\begin{document}{\allowdisplaybreaks[4]

\title[{Pseudotoric structures and special Lagrangian fibrations}]{Pseudotoric structures and special Lagrangian torus fibrations on certain flag varieties}

\author[K. Chan]{Kwokwai Chan}
\address{Department of Mathematics, The Chinese University of Hong Kong, Shatin, Hong Kong}
\email{kwchan@math.cuhk.edu.hk}

\author[N. C. Leung]{Naichung Conan Leung}
\address{The Institute of Mathematical Sciences and Department of Mathematics, The Chinese University of Hong Kong, Shatin, Hong Kong}
\email{leung@math.cuhk.edu.hk}

\author[C. Li]{Changzheng Li}
\address{School of Mathematics, Sun Yat-sen University, Guangzhou 510275, P.R. China;}
\email{lichangzh@mail.sysu.edu.cn}





\begin{abstract}
We construct pseudotoric structures (\`{a} la Tyurin \cite{Tyu}) on the two-step flag variety $F\ell_{1, n-1; n}$, and explain a general relation between pseudotoric structures and special Lagrangian torus fibrations, the latter of which are important in the study of SYZ mirror symmetry \cite{SYZ}. As an application, we speculate how our constructions can explain the number of terms in the superpotential of Rietsch's Landau-Ginzburg mirror \cite{Riet, MaRi, PeRi, PRW}.
\end{abstract}

\maketitle

\tableofcontents

\section{Introduction}

A {\bf pseudotoric structure} on a K\"ahler manifold $X$ of complex dimension $n$ consists of
a Hamiltonian $T^k$-action on $(X, \omega_X)$, where $\omega_X$ is the K\"ahler structure on $X$ and $1 \leq k \leq n$, together with
a meromorphic map $F: X \dashrightarrow Y$ from $X$ to a smooth projective toric variety $Y$ of complex dimension $n-k$
satisfying certain compatibility conditions.
This notion, due to Tyurin \cite{Tyu}, is a generalization of toric varieties and was motivated by Auroux's pioneering work \cite{Aur1, Aur2} on SYZ mirror symmetry \cite{SYZ}.

In his original paper \cite{Tyu}, Tyurin explicitly constructed
pseudotoric structures on
a smooth quadric hypersurface $Q$ in $\mathbb{P}^n$.
In this note, we will do the same for
the two-step flag variety
$$F\ell_{1, n-1; n} = \{W_1\leqslant W_{n-1}\leqslant \mathbb{C}^n \mid \dim W_1 = 1, \dim W_{n-1} = n-1\},$$
which is a degree $(1,1)$ hypersurface in
$$\mathbb{P}\left(\wedge^1 \mathbb{C}^n\right) \times \mathbb{P}\left(\wedge^{n-1} \mathbb{C}^n\right) \cong \mathbb{P}^{n-1} \times \mathbb{P}^{n-1};$$
see Theorem \ref{thm:pseudotoric_two-step_flag}.

Pseudotoric structures are closely related to SYZ mirror symmetry because, as will be proved in this note,
when there is a meromorphic top form $\Omega_X$ on $X$ with no zeros and only simple poles along an anticanonical divisor $D \subset X$ satisfying the condition
\begin{equation}\label{eqn:condition_holom_vol_form_intro}
\iota_{V_1}\cdots \iota_{V_k} \Omega_X = d\log F^*f_1 \wedge \cdots \wedge d\log F^*f_{n-k},
\end{equation}
where $\{V_1,\dots,V_k\}$ is a basis of Hamiltonian vector fields generating the $T^k$-action and $f_1, \ldots, f_{n-k}$ are torus invariant rational functions on $Y$ whose norms are Poisson commuting with respect to the Poisson bracket induced by the toric symplectic sturcture on $Y$, then
$$\rho := (\mu_{T^k}; F^*|f_1|, \ldots, F^*|f_{n-k}|): X \setminus D \to B \subset \mathbb{R}^d$$
defines a special Lagrangian torus fibration on $X \setminus D$, and the SYZ construction \cite{SYZ} may be applied to it to construct a mirror.

We show that meromorphic top forms satisfying \eqref{eqn:condition_holom_vol_form_intro} can be found on both
\begin{itemize}
\item
smooth quadric hypersurfaces $Q$ in $\mathbb{P}^n$ (in Section \ref{sec:Omega_quadric}), and
\item
the two-step flag variety $F\ell_{1, n-1; n}$ (in Section \ref{sec:Omega_two-step_flag}),
\end{itemize}
with suitable choices of the anticanonical divisor $D$, giving rise to various special Lagrangian torus fibrations (see Theorems \ref{thm:slag_two-step_flag}, \ref{thm:slag_odd-dim_quadric} and \ref{thm:slag_odd-dim_quadric}).
The simplest examples in these two classes are
the Grassmannian $Gr(2,4)$ and
the full flag variety $F\ell_3$ respectively.

In \cite{Riet}, Rietsch proposed a Lie-theoretic construction of a Landau-Ginzburg (abbrev. as LG) mirror $W^{\text{Rie}}$ for any flag variety $X = G/P$ with the following properties:
\begin{itemize}
\item
The number of terms in $W^{\text{Rie}}$ is less than that of Givental's mirror $W^{\text{Giv}}$.
\item
The superpotential $W^{\text{Rie}}$ is defined on a space bigger than $(\mathbb{C}^*)^{\dim X}$ and has the correct number of critical points, while $W^{\text{Giv}}$ is defined on $(\mathbb{C}^*)^{\dim X}$ and may not have enough critical points (e.g. in the case of $Gr(2,4)$).
\end{itemize}
A natural question is:

\begin{center}
{\em How to understand Rietsch's mirror using SYZ?}\\
\end{center}

We first notice that for both the quadric $Q$ and the two-step flag variety $F\ell_{1, n-1; n}$, there are smooth families of integrable systems connecting
the special Lagrangian torus fibration $\rho: X \setminus D \to B$ constructed above with
the Gelfand-Cetlin fibration (restricted to the preimage of the interior of the Gelfand-Cetlin polytope).
We expect that: {\em $W^{\text{Rie}}$ is obtained from $W^{\text{Giv}}$ by ``merging'' some of the terms which geometrically corresponds to {\bf gluing of holomorphic disks} and a coordinate change coming from the relevant {\bf wall-crossing formula}}.

Evidence for this claim will be provided for both $Q$ and $F\ell_{1, n-1; n}$.
In the discussion in Section \ref{sec:speculations}, we speculate that, along the smooth family of integrable systems, facets of the Gelfand-Cetlin polytope are being ``pushed into'' the interior to become walls in the base $B$. Geometrically, this corresponds to deforming (or partially smoothing of) the anticanonical divisor $D \subset X$; and accordingly, terms in the superpotential $W^{\text{Giv}}$ are being merged together (as, conjecturally, holomorphic disks are being glued together) to give terms in the new superpotential.
For the eventual superpotential, the number of terms (which should be less than that of $W^{\text{Giv}}$) should coincide with that of $W^{\text{Rie}}$.

As an example, let us consider $X = F\ell_{1, n-1; n}$.
There are $2n$ terms in Givental's superpotential $W^{\text{Giv}}$ corresponding to the $2n$ facets in the Gelfand-Cetlin polytope (see e.g. \cite{NNU00}),
4 of which have preimages which are {\em not} algebraic subvarieties in $X$.
When we move from the Gelfand-Cetlin system to our fibration along the smooth family of integrable systems, 2 of the facets are being ``pushed in'' to form a wall;
accordingly, 4 terms in $W^{\text{Giv}}$ are merged to give 2 new terms.
For the remaining $2n - 4$ facets, $2n - 6$ of them come in pairs and each pair is being ``pushed in'' to form a wall;
accordingly, $2(n-3)$ terms in $W^{\text{Giv}}$ are being merged in pairs to produce $n - 3$ new terms.
So the eventual LG superpotential should have
$$2 + (n - 3) + 2 = n + 1$$
terms, which is expected to be the number of terms in Rietsch's superpotential $W^{\text{Rie}}$. As we will see in Remark \ref{rmkRietmirror}, such expectation holds   for $n = 3, 4$.

Recently, the above speculations have been realized by the work of Hong, Kim and Lau \cite{HKL}, at least for the case of $Gr(2,n)$ (and also for $OG(1,5)$). The idea is that, by deforming the Gelfand-Cetlin fibration to the special Lagrangian torus fibration $\rho$,\footnote{In fact, Hong, Kim and Lau applied the technique in Abouzaid-Auroux-Katzarkov's work \cite{AAK}, instead of Tyurin's pseudotoric structures, to construct the Lagrangian torus fibrations, but their fibrations have essentially the same structures as the ones we constructed here.} the non-torus fibers of the Gelfand-Cetlin fibration, which are known to be non-trivial objects in the Fukaya category \cite{NoUe}, are deformed to certain {\it immersed} Lagrangians.

In the case of $Gr(2,n)$, it has already been shown by Nohara and Ueda \cite{NoUe2} that potential functions of the Lagrangian torus fibers of {\it different} Gelfand-Cetlin--type fibrations (which correspond to triangulations of the regular $n$-gon) can be glued (via cluster transformations) to give an open dense subset of Marsh-Rietsch's mirror \cite{MaRi}. But this is still insufficient as the open dense subset misses some of the critical points of Marsh-Rietsch's mirrors.

What Hong, Kim and Lau showed in \cite{HKL} was that deformations of the immersed Lagrangian above (a local model of which was given in \cite{CPU}) produces the final missing chart in Marsh-Rietsch's mirror of $Gr(2,n)$. This completes the SYZ construction, namely, by applying SYZ (or family Floer theory) to regular as well as singular fibers of the special Lagrangian torus fibration $\rho$, one can recover Marsh-Rietsch's mirror. Similar constructions should be applicable to the two-step flag varieties $F\ell_{1, n-1; n}$ \cite{HKL2}.

\subsection*{Acknowledgment}
We thank Yoosik Kim and Siu-Cheong Lau for detailed explanations of their work and very useful discussions. We are also indebted to the anonymous referee for insightful comments and suggestions which helped to improve the exposition of this article.

K. Chan was supported by Hong Kong RGC grant CUHK14300314 and direct grants from CUHK.
N. C. Leung was supported by Hong Kong RGC grants CUHK14302215 $\&$ CUHK14303516 and direct grants from CUHK.
C. Li was supported by NSFC grants 11822113, 11831017 and 11521101.

\section{Pseudotoric structures (after N. A. Tyurin)}


Let $X$ be a K\"ahler manifold of complex dimension $n$ with K\"ahler structure $\omega_X$.

\begin{defn}[cf. Definition 1 in \cite{Tyu}]\label{defn:pseudotoric_structures}
A {\bf pseudotoric structure} on $X$ consists of the following data:
\begin{itemize}
\item[(1)]
a Hamiltonian $T^k$-action on $(X, \omega_X)$, where $1 \leq k \leq n$, and

\item[(2)]
a meromorphic map $F: X \dashrightarrow Y$ from $X$ to a toric K\"ahler manifold $Y$ of complex dimension $n-k$ equipped with a toric K\"ahler structure $\omega_Y$,
\end{itemize}
satisfying the following conditions:
\begin{itemize}
\item[(i)]
the base locus of $F$ is a $T^k$-invariant subvariety $B$ of complex dimension $k-1$ in $X$,

\item[(ii)]
the holomorphic map $F: X \setminus B \to Y$ is $T^k$-invariant, and

\item[(iii)]
for any smooth function $h$ on $Y$, we have the relation
\begin{equation}\label{eqn:condition_Ham_vec_field}
V_{F^*h} = f \cdot \nabla_F V_h,
\end{equation}
where $V_{F^*h}$ and $V_h$ are the Hamiltonian vector fields of the functions $F^*h$ and $h$ on $X \setminus B$ and $Y$ with respect to $\omega_X$ and $\omega_Y$ respectively, $f$ is a real function on $X \setminus B$, $\nabla_F$ is the symplectic connection induced by the symplectic fibration $F$, and $\nabla_F V_h$ denotes the horizontal lift of the vector field $X_h$ to $X \setminus B$.
\end{itemize}
The complex dimension of $Y$ is called the {\bf rank} of the pseudotoric structure.
\end{defn}

\begin{remark}
Our definition is slightly different from that of Tyurin \cite[Definition 1]{Tyu}, which only assumes that $X$ is a symplectic manifold.
\end{remark}


By fixing a basis $\{v_1,\dots,v_k\}$ of the Lie algebra $\text{Lie}(T^k)$ (which gives rise to a basis $\{V_1,\dots,V_k\}$ of Hamiltonian vector fields), we can write the moment map associated to the Hamiltonian $T^k$-action on $(X, \omega_X)$ as
$$\mu = (\mu_1, \ldots, \mu_k): X \to \mathbb{R}^k \cong \text{Lie}(T^k)^*,$$
i.e. $\iota_{V_i}\omega_X = d\mu_i$ for $i = 1,\dots, k$.

\begin{lemma}
For $1  \leq i \leq k$ and for any smooth function $h$ on $Y$, we have
$$\{ \mu_i, F^*h \} = \omega_X (V_i, V_{F^*h}) = 0,$$
where $\{\cdot, \cdot\}$ denotes the Poisson bracket on $X$ induced by $\omega_X$.
\end{lemma}
\begin{proof}
Since $F: X \setminus B \to Y$ is a $T^k$-invariant symplectic fibration, $V_i$ is tangent to the fibers of $F$ for each $i$. On the other hand, for any vector field $W$ lying in the vertical subbundle $\text{Ver} := \text{ker} (dF: T(X\setminus B) \to TY)$, we have
$$\omega_X(V_{F^*h}, W) = d(F^*h)(W) = (F^*(dh))(W) = dh(dF(W)) = 0.$$
The result follows.
\end{proof}

Let $\nu_1, \dots, \nu_{n-k}$ be functions on the image $U \subset Y$ of the holomorphic map $F: X \setminus B \to Y$ which are {\em Poisson commuting} with respect to the Poisson bracket induced by the toric symplectic structure $\omega_Y$ on $Y$. For example, such functions are given by the components of the toric moment map
$$\nu = (\nu_1, \ldots, \nu_{n-k}): Y \to \mathbb{R}^{n-k} \cong \text{Lie}(T^{n-k})^*$$
associated to the toric manifold $Y$.

\begin{lemma}
For $1 \leq j, \ell \leq n-k$, we have
$$\{ F^*\nu_j, F^*\nu_\ell \} = \omega_X (V_{F^*\nu_j}, V_{F^*\nu_\ell}) = 0.$$
\end{lemma}
\begin{proof}
By the condition \eqref{eqn:condition_Ham_vec_field}, we have
$$\omega_X (V_{F^*\nu_j}, V_{F^*\nu_\ell}) = d(F^*\nu_j) (f \nabla_F V_{\nu_\ell}) = f F^* (d\nu_j (V_{\nu_\ell})) = f F^*\{\nu_j, \nu_\ell\} = 0.$$
\end{proof}

These lemmas prove the following

\begin{prop}\label{prop:Lagrangian_fibration}
The functions
$$\mu_1, \ldots, \mu_k, F^*\nu_1, \ldots, F^*\nu_{n-k}$$
form a set of Poisson commuting algebraically independent functions over $X \setminus B$. Therefore the map
$$\rho := (\mu_1, \ldots, \mu_k, F^*\nu_1, \ldots, F^*\nu_{n-k}): X \setminus B \to \mathbb{R}^n$$
defines a completely integrable system on $X \setminus B$.
\end{prop}



Pseudotoric structures are difficult to find in general. But in \cite{Tyu}, Tyurin gave a (rather restrictive) sufficient condition, which we reformulate as follows.

For any smooth function $h$ on $Y$, the Hamiltonian vector fields $V_h$ and $V_{F^*h}$ are respectively defined by
$$\omega_Y(V_h, \cdot) = dh(\cdot) \text{ and } \omega_X(V_{F^*h}, \cdot) = (d(F^*h))(\cdot).$$
Applying the pullback $F^*$ to the former equation gives
$$\omega_Y(V_h, dF(\cdot)) = (d(F^*h))(\cdot)$$
as 1-forms on $X$.
On the other hand, by the definition of the horizontal lift, we have $dF(\nabla_F V_h) = V_h$ and hence
$$(F^*\omega_Y)(\nabla_F V_h, \cdot) = \omega_Y(V_h, dF(\cdot)).$$
So altogether we have
$$\omega_X(V_{F^*h}, \cdot) = (F^*\omega_Y)(\nabla_F V_h, \cdot).$$
Hence, a sufficient condition to guarantee that condition \eqref{eqn:condition_Ham_vec_field} holds is the following:
\begin{equation*}
\textit{The 1-forms $\omega_X(\nabla_F V_h, \cdot)$ and $(F^*\omega_Y)(\nabla_F V_h, \cdot)$ are parallel to each other over $X \setminus B$.}
\end{equation*}
Notice that both $\omega_X(\nabla_F V_h, \cdot)$ and $(F^*\omega_Y)(\nabla_F V_h, \cdot)$ annihilate the vertical subbundle $\text{Ver} \subset TX$, so they are determined by their values on the horizontal subbundle $\text{Hor}$.

\begin{prop}
If $Y$ is of complex dimension one, then condition \eqref{eqn:condition_Ham_vec_field} is automatically satisfied.
\end{prop}
\begin{proof}
When $\dim_{\mathbb{C}} Y = 1$, the horizontal subbundle $\text{Hor}$ is of rank 2, so $\bigwedge^2 \text{Hor}^*$ is of rank 1. Hence $\omega_X(\cdot, \cdot)$ and $(F^*\omega_Y)(\cdot, \cdot)$ only differ by a real function over $\text{Hor}$.
\end{proof}
This is why the condition \eqref{eqn:condition_Ham_vec_field} did not appear in the cohomogeneity one examples in \cite{Aur1, Aur2, CLL}.

\begin{example}\label{eg:key_example}
Suppose that both $X$ and $Y$ are submanifolds of a complex projective space $\mathbb{P}^N$ and their symplectic structures $\omega_X$ and $\omega_Y$ are restrictions of the Fubini-Study form $\omega_{\text{FS}}$ on $\mathbb{P}^N$. Also suppose that the map $F: X \to Y$ is the restriction of a projection map $\tilde{F}: \mathbb{P}^N \to \mathbb{P}^N$ onto a linear subspace $L \subset \mathbb{P}^N$. Then, by a linear change of coordinates if necessary, we may assume that $L$ is given by the common zero set of a subset of the coordinates. In this case, a straightforward computation shows that $\omega_X(\nabla_F V_h, \cdot)$ is parallel to $(F^*\omega_Y)(\nabla_F V_h, \cdot)$ over the horizontal subbundle $\text{Hor}$ (i.e., differ by a constant at every point $x \in X$), and so condition \eqref{eqn:condition_Ham_vec_field} holds.
\end{example}

\begin{lemma}\label{lem:restriction}
Condition \eqref{eqn:condition_Ham_vec_field} remains valid after restriction to a sub-fibration. More precisely, if $\tilde{F}: M \to N$ is a symplectic fibration for which condition \eqref{eqn:condition_Ham_vec_field} holds, and if $F: X \to Y$ is a symplectic fibration such that $X$ and $Y$ are symplectic submanifolds of $M$ and $N$ respectively, and $F$ is the restriction of $\tilde{F}$ to $X$, i.e., the following diagram commutes
\begin{equation*}
\xymatrix{
X \ar@{->>}[d]_{F} \ar@{^{(}->}[r]^{\iota_X} & M \ar@{->>}[d]^{\tilde{F}}\\
Y \ar@{^{(}->}[r]^{\iota_Y} & N}
\end{equation*}
Then condition \eqref{eqn:condition_Ham_vec_field} also holds for $F: X \to Y$ for any function on $Y$ coming from the restriction of a function of $N$.
\end{lemma}
\begin{proof}
The result follows from the following two observations:
\begin{itemize}
\item[(1)]
For a symplectic submanifold $\iota_X: X \hookrightarrow M$ and a smooth function $f:M \to \mathbb{R}$, the Hamiltonian vector field $V_{\iota_X^*f}$ of the restriction $\iota_X^*f$ is given by restricting the Hamiltonian vector field $V_f$ to $X$ and projecting to the subbundle $TX \subset TM$.

\item[(2)]
On the other hand, since $\iota_X(F^{-1}(y)) = \iota_X(X) \cap (\tilde{F})^{-1}(\iota_Y(y))$, if the decomposition of $TM$ into the direct sum of the vertical and horizontal subbundles with respect to the fibration $\tilde{F}:M \to N$ is given by
$$TM = \text{Ver} \oplus \text{Hor},$$
then the corresponding decomposition of $TX$ with respect to $F:X \to Y$ is precisely
$$TX = (\text{Ver} \cap TX) \oplus (\text{Hor} \cap TX).$$
Also, $\tilde{F} \circ \iota_X = \iota_Y \circ F$, so the horizontal lift $\nabla_F V_{\iota_Y^*h}$ is given by restricting the horizontal lift $\nabla_{\tilde{F}} V_h$ to $X$ and projecting to the subbundle $\text{Hor} \cap TX \subset \text{Hor}$.
\end{itemize}
\end{proof}

\subsection{The two-step flag variety $F\ell_{1,n-1; n}$}

The two-step flag variety
$$F\ell_{1,n-1; n} = \{W_1\leqslant W_{n-1}\leqslant \mathbb{C}^n \mid \dim W_1 = 1, \dim W_{n-1} = n-1\}$$
is a homogeneous variety of the form $SL(n,\mathbb{C})/P$.
Denote by $[x_1: \cdots: x_n]$ the homogenous coordinates of $\mathbb{P}(\bigwedge^1 \mathbb{C}^n)$, and by $[x_{\hat 1}: \cdots: x_{\hat n}]$ the homogeneous coordinates of $\mathbb{P}(\bigwedge^{n-1} \mathbb{C}^n)$; here we are using the notational convention
$$x_{\hat j}:= x_{1 2 \cdots (j-1)(j+1) \cdots n}.$$

Then the well-known {\em Pl\"ucker embedding} of $X=F\ell_{1, n-1; n}$ in $\mathbb{P}(\bigwedge^1 \mathbb{C}^n) \times \mathbb{P}(\bigwedge^{n-1} \mathbb{C}^n) \cong \mathbb{P}^{n-1} \times \mathbb{P}^{n-1}$ is given by
$$x_{2}x_{\hat 2} = x_1x_{\hat 1} + \sum_{j=3}^n(-1)^{j-1}x_jx_{\hat j}.$$
In particular, it is of complex dimension $2n-3$. In terms of the inhomogeneous coordinates
$$z_j = {x_j \over x_1}, z_{\hat j} = {x_{\hat j} \over x_{\hat n}}$$
in the open subset $x_1\neq 0, x_{\hat n}\neq 0$, it is given by
$$z_{2}z_{\hat 2}= z_{\hat 1}+(-1)^{n-1}z_n+\sum_{j=3}^{n-1}(-1)^{j-1}z_jz_{\hat j}.$$

The maximal complex torus
$$T_{\mathbb{C}}=\{\mathbf{t} = (t_1, \ldots, t_n)\in (\mathbb{C}^*)^n \mid t_1 t_2 \cdots t_n = 1\},$$
as a subgroup of $SL(n, \mathbb{C})$ and of dimension $n-1$, and hence induces an action on
$\mathbb{P}(\bigwedge^1 \mathbb{C}^n) \times \mathbb{P}(\bigwedge^{n-1} \mathbb{C}^n)$ given by
  \begin{equation}\label{actiontwostep}
     \mathbf{t}\cdot ([x_1:\cdots: x_n],[x_{\hat 1}:\cdots:x_{\hat n}])=([t_1x_1:\cdots: t_nx_n],[t_{\hat 1}x_{\hat 1}:\cdots:t_{\hat n}x_{\hat n}])
  \end{equation}  where we denote $t_{\hat j}:= t_{1} t_2 \cdots t_{j-1}t_{j+1} \cdots t_n.$ 
Clearly, the $T_\mathbb{C}$-action preserves $X$. Moreover,  the induced action of the maximal compact torus $T\subset T_{\mathbb{C}}$ on $X$,
 where $\mathbf{t}=(\exp^{\sqrt{-1} \theta_1},\cdots, \exp^{\sqrt{-1}\theta_n})\in T$,   is Hamiltonian with moment map
$$\mu = (\mu_1,\ldots, \mu_{n-1}): X\longrightarrow \mathbb{R}^{n-1}.$$

We define a rational map
$$F: F\ell_{1,n-1; n} \dashrightarrow \mathbb{P}^{n-2}$$
by setting
$$y_0 = x_1x_{\hat 1}, y_1 = x_3x_{\hat 3}, y_2 = x_4x_{\hat 4}, \ldots, y_{n-2} = x_nx_{\hat n}.$$
The base locus of $F$ is given by the Schubert subvariety $B$ defined by $x_1x_{\hat 1} = x_3x_{\hat 3} = \cdots = x_nx_{\hat n} = 0$, which is of complex dimension $n-2$ in $F\ell_{1,n-1; n}$.

\begin{thm}\label{thm:pseudotoric_two-step_flag}
This defines a pseudotoric structure on the two-step flag variety $F\ell_{1,n-1; n}$.
\end{thm}
\begin{proof}
To see that this produces a pesudotoric structure, we embed $\mathbb{P}^{n-1} \times \mathbb{P}^{n-1}$ and hence $F\ell_{1,n-1; n}$ into $\mathbb{P}^N$, where $N = n^2 - 1$, by the Segre embedding. We also embed $\mathbb{P}^{n-2}$ into the same $\mathbb{P}^N$ in a natural way so that $F$ is the restriction of a projection map $\tilde{F}: \mathbb{P}^N \to \mathbb{P}^N$ onto a linear subspace (image of $\mathbb{P}^{n-2}$). Then the result follows from Example \ref{eg:key_example} and Lemma \ref{lem:restriction}.
\end{proof}

\subsection{The quadrics}
This example was treated in Tyurin \cite[Theorem 2]{Tyu}, where he showed that an arbitrary smooth quadric hypersurface in $\mathbb{P}^n$ admits a pseudotoric structure.

For $n = 2m$, consider the quadric
$$Q = \{ x_0^2 + x_1x_2 + \cdots + x_{2m-1}x_{2m} = 0 \}.$$
This quadric is a homogeneous variety $G/P$ for $G=SO(2m+1, \mathbb{C})$. A maximal torus $T^{m}_{\mathbb{C}}$ of $SO(2m+1, \mathbb{C})$ acts on the quadric by
\begin{equation}\label{actionoddqua}
   (t_1, \ldots, t_{m})\cdot [x_0:x_1:\cdots: x_{2m}] = [x_0, t_1x_1, t_1^{-1}x_2, \ldots, t_mx_{2m-1}, t_{m}^{-1}x_{2m}].
\end{equation}
Define the rational map
$$F: \mathbb{P}^{2m} \dashrightarrow \mathbb{P}^m$$
by setting
$$y_0 = x_0^2, y_1 = x_1x_2, \ldots, y_m = x_{2m-1}x_{2m}.$$
This restricts to $Q$ to give a $T^m$-invariant rational map
$$F: Q \dashrightarrow \mathbb{P}^{m-1},$$
where $T^m$ denotes the standard maximal compact subtorus of $T^{m}_{\mathbb{C}}$, and $\mathbb{P}^{m-1} \subset \mathbb{P}^m$ is defined by
$$y_0 + y_1 + \cdots + y_m = 0.$$
The base locus of $F$ is given by the union of a set of linear subspaces of dimension $m$ in $\mathbb{P}^{2m}$, so the base locus $B$ is of complex dimension $m-1$ in $Q$.

Similarly, when $n = 2m + 1$, consider the quadric
$$Q = \{ x_0x_1 + x_2x_3 + \cdots + x_{2m}x_{2m+1} = 0 \}.$$
This quadric is a homogeneous variety $G/P$ for $G=SO(2m+2, \mathbb{C})$. A maximal complex torus $T^{m+1}_{\mathbb{C}}$ of $SO(2m+2, \mathbb{C})$ acts on the quadric by
\begin{equation}\label{actionevenqua}
   (t_0, \ldots, t_{m})\cdot  [x_0:x_1:\cdots: x_{2m+1}] = [t_0x_0, t_0^{-1}x_1, t_1x_2, t_1^{-1}x_3, \ldots, t_mx_{2m}, t_{m}^{-1}x_{2m+1}].
\end{equation}
Define the rational map
$$F: \mathbb{P}^{2m+1} \dashrightarrow \mathbb{P}^m$$
by setting
$$y_0 = x_0x_1, y_1 = x_2x_3, \ldots, y_m = x_{2m}x_{2m+1}.$$
This restricts to $Q$ to give a $T^{m+1}$-invariant rational map
$$F: Q \dashrightarrow \mathbb{P}^{m-1},$$
where  $T^{m+1}$ denotes the standard maximal compact subtorus of $T^{m+1}_{\mathbb{C}}$, $\mathbb{P}^{m-1} \subset \mathbb{P}^m$ is defined by
$$y_0 + y_1 + \cdots + y_m = 0.$$
The base locus of $F$ is given by the union of a set of linear subspaces of dimension $m$ in $\mathbb{P}^{2m+1}$, so the base locus $B$ is of complex dimension $m-1$ in $Q$.

\begin{thm}[Theorem 2 in \cite{Tyu}]\label{thm:pseudotoric_quadric}
This defines a pseudotoric structure on the smooth quadric.
\end{thm}
\begin{proof}
The proof is very similar to the previous example, but this time we embed $\mathbb{P}^{2m}$ into $\mathbb{P}^N$, where $N = {2m+2 \choose 2} - 1$, by the Veronese embedding, and also embed $\mathbb{P}^m$ into the same $\mathbb{P}^N$ in a natural way. Then we see that $F$ comes from the restriction of a projection map $\tilde{F}: \mathbb{P}^N \to \mathbb{P}^N$ onto a linear subspace (image of $\mathbb{P}^m$). So the result again follows from Example \ref{eg:key_example} and Lemma \ref{lem:restriction}.
\end{proof}

\section{Special Lagrangian torus fibrations}\label{sec:sLag_fibration}

In this section, we explain how pseudotoric structures are related to special Lagrangian fibrations, and hence SYZ mirror symmetry \cite{SYZ}. Recall that a submanifold $L \subset X$ is called {\em special Lagrangian} if it is Lagrangian and $\text{Im } e^{\sqrt{-1}\theta}\Omega_X|_L = 0$ for some $\theta \in \mathbb{R}$. We have the following proposition:

\begin{prop}\label{prop:sLag_fibration}
Let $(X, \omega_X)$ be a K\"ahler manifold equipped with a pseudotoric structure, i.e. with a Hamiltonian $T^k$-action on $(X, \omega_X)$ for some $1\leq k \leq n$ and a meromorphic map $F: X \dashrightarrow Y$ from $X$ to a toric K\"ahler manifold $Y$ of complex dimension $n-k$ equipped with a toric K\"ahler structure $\omega_Y$ satisfying the conditions in Definition \ref{defn:pseudotoric_structures}.
Suppose that there is an anticanonical divisor $D \in |-K_X|$ and a meromorphic $n$-form $\Omega_X$ on $X$ which has no zeros and only simple poles along the divisor $D$, satisfying the following conditions:
\begin{itemize}
\item[(i)]
$D$ is $T^k$-invariant.

\item[(ii)]
The holomorphic map $F: X \setminus B \to Y$ restricts to a map $F: X \setminus D \to Y \setminus E$, where $E$ is the toric anticanonical divisor in $Y$.

\item[(iii)]
Given a basis $\{V_1,\dots,V_k\}$ of Hamiltonian vector fields generating the $T^k$-action on $(X, \omega_X)$, we have
\begin{equation}\label{eqn:condition_holom_vol_form}
\iota_{V_1}\cdots \iota_{V_k} \Omega_X = d\log F^*f_1 \wedge \cdots \wedge d\log F^*f_{n-k}
\end{equation}
for some $T^k$-invariant meromorphic functions $f_1, \ldots, f_{n-k}$ on $Y$.\footnote{Note that the LHS of \eqref{eqn:condition_holom_vol_form} differs only by a scalar multiple for different choices of the basis $\{V_1,\dots,V_k\}$, so this condition is independent of such a choice.}

\item[(iv)]
The functions $\nu_j := |f_j|$, $j = 1, \dots, n-k$ on the open dense orbit $Y_0 := Y\setminus E \cong (\mathbb{C}^*)^{n-k}$ are Poisson commuting with respect to the Poisson bracket induced by the toric symplectic structure $\omega_Y$ on $Y$.
\end{itemize}
Then the map
$$\rho := (\mu_1, \ldots, \mu_k, F^*\nu_1, \ldots, F^*\nu_{n-k}): X \setminus D \to \mathbb{R}^n$$
defines a special Lagrangian fibration on $X \setminus D$, where the special condition is with respect to the holomorphic volume form $\Omega_X$ on $X\setminus D$.
\end{prop}
\begin{proof}
Proposition \ref{prop:Lagrangian_fibration} already shows that $\rho$ is a Lagrangian torus fibration. So we only need to prove that the fibers are special with respect to $\Omega_X$.

The following argument is along the lines of \cite[Proof of Proposition 5.2]{Aur1}.
First note that the torus $T^k$ acts on $X$ by symplectomorphisms and the holomorphic map $F: X\setminus B \to Y$ is $T^k$-invariant. Hence the parallel transport induced by the symplectic connection $\nabla_F$ is $T^k$-equivariant. In particular, any Lagrangian fiber $L$ of $\rho$ is invariant under the parallel transport along the corresponding fiber of the map $\nu = (\nu_1,\dots,\nu_{n-k}): Y_0 \to \mathbb{R}^{n-k}$. Thus the horizontal lifts $\nabla_F V_{|f_j|}$ ($j = 1, \ldots, n-k$) are tangent to $L$. But condition \eqref{eqn:condition_Ham_vec_field} says that each $\nabla_F V_{|f_j|}$ is parallel to $V_{F^*|f_j|}$. Therefore, the tangent space to any point on $L$ is spanned by the vector fields $V_1, \ldots, V_k$ (which generate the $T^k$-action) and $V_{F^*|f_1|}, \ldots, V_{F^*|f_{n-k}|}$.

Now at every point on the Lagrangian fiber $L$, we have, by condition (iii) or more precisely equation \eqref{eqn:condition_holom_vol_form}, that
\begin{align*}
  & \Omega_X|_L (V_1, \ldots, V_k, V_{F^*|f_1|}, \ldots, V_{F^*|f_{n-k}|}) \\
= & (\iota_{V_1}\cdots \iota_{V_k} \Omega_X)|_L (V_{F^*|f_1|}, \ldots, V_{F^*|f_{n-k}|})\\
= & (d\log F^*f_1 \wedge \cdots \wedge d\log F^*f_{n-k})|_L (V_{F^*|f_1|}, \ldots, V_{F^*|f_{n-k}|}).
\end{align*}
Since $V_{F^*|f_j|}$ is tangent to the level sets of $F^*|f_j|$, we have either $\text{Im } \Omega_X|_L = 0$ (when $n-k$ is even) or $\text{Re }\Omega_X|_L = 0$ (when $n-k$ is odd).
\end{proof}

\subsection{The two-step flag variety $F\ell_{1,n-1; n}$}\label{sec:Omega_two-step_flag}

In this subsection, we show that the condition \eqref{eqn:condition_holom_vol_form} is satisfied for the pseudotoric structure on the two-step flag variety $F\ell_{1,n-1; n}$ we constructed in Theorem \ref{thm:pseudotoric_two-step_flag}.

Take any anti-canonical divisor $D \in |-K_X|$ of $X$, we obtain an open Calabi-Yau manifold $X\setminus D$. In this section, we will specify several choices of $D$, and then construct special Lagrangian fibrations on $X\setminus D$ with respect to a meromorphic volume form $\Omega_{X}$ on $X$ that has no zeros and only simple poles along $D$. One of such choices will be given by Schubert divisors; we refer to \cite[Lemma 3.5]{FuWo} for the description of the anti-canonical divisor class $[-K_X]$ by Schubert divisor classes for a general homogeneous variety $G/P$. We notice that the two-step flag variety $F\ell_{1,n-1; n}$ and the quadrics are all special cases of such homogeneous varieties.

Recall that the natural action of
$T=\{\text{diag}(t_1, \ldots, t_n)\in (\mathbb{S}^1)^n \mid t_1t_2\cdots  t_n=1\}$
(as a maximal compact torus of the maximal complex torus $T_\mathbb{C}$ of $SL(n, \mathbb{C})$)   on $X$ was given in \eqref{actiontwostep}.
We shall specify  a basis $V_1, \ldots, V_{n-1}$ of Hamiltonian vector fields generated by the torus action of $T$.
As we will see later, our choices of $\Omega_X$ are very nice in the sense that the condition
$$\iota_{V_1}\cdots \iota_{V_{n-1}}\Omega_X =\pm d\log f_1\wedge\cdots d\log f_{n-2}$$
is satisfied, and from this we can construct a special Lagrangian torus fibration (with singular fibers)
$$\Phi:=(\mu_1, \ldots, \mu_{n-1}, |f_1|, \ldots, |f_{n-2}|): X\setminus D \longrightarrow \mathbb{R}^{2n-3},$$
where $(\mu_1,\ldots, \mu_{n-1}): X\longrightarrow \mathbb{R}^{n-1}$ is the moment map of the torus action of $T$.

\subsubsection{Contraction by torus-invariant holomorphic vector fields}
On the open subset $x_1\neq 0, x_{\hat n}\neq 0$, we have local coordinates $(z_2, z_3,\ldots, z_n, z_{\hat 2},\ldots, z_{\widehat{n-1}})$ of $X$.

Let $\theta_k$ denote the weight of $t_k$. Then for $1\leq j\leq n-1$, the weight of the action of $T$ on the coordinate $z_{j+1}$  is given by
$$\psi_{j}:=\theta_{j+1}-\theta_1.$$
For $2\leq k\leq n-1$, the weight of the action of $T$ on the coordinate $z_{\hat k}$ is given by
$$\theta_{n}-\theta_k=\psi_{n-1}-\psi_{k-1}.$$

Let $V_i$ be the Hamiltonian vector field corresponding to the weight $\psi_i$. Denote by $c_y$ the coefficient of $\psi_i$ in the weight for a coordinate. Then we have 
$$\iota_{V_i}=\sum_y c_y\cdot y \cdot \iota_{\partial_y},$$
where $\iota_{\partial_y}$ denotes the contraction with the holomorphic vector field ${\partial\over \partial y}$ (on the left). We simply denote $\iota_{\partial_y}$ as $\iota_y$. We also adapt the notation convention that $z_1=1, z_{\hat n}=1$, and denote
$$\Omega:=dz_2dz_{\hat 2}\cdots dz_{n-1}dz_{\widehat{n-1}}dz_n,\quad A_j:=z_jz_{\hat j},\ j=1,\dots, n.$$

\begin{lemma}\label{lem:contrusual}
$$ \iota_{V_1}\cdots \iota_{V_{n-1}}\Omega= (-1)^{n} A_1dA_3\cdots  dA_n+\sum_{j=3}^{n}(-1)^{n+j} A_jdA_1dA_3\cdots \widehat{dA_j}\cdots dA_n.$$
\end{lemma}
\begin{proof}
We notice that $\iota_{V_{i-1}}=z_i\iota_{z_i}-z_{\hat i}\iota_{z_{\hat{i}}}$ for $i=2,\dots, n-1$, and that
$\iota_{V_{n-1}}=z_n\iota_{z_5}+\sum_{j=2}^{n-1} z_{\hat j}\iota_{z_{\hat j}}$.
Thus for $2\leq i\leq n-1$, we have $\iota_{V_{i-1}} dz_idz_{\hat i}= z_i dz_i+z_{\hat i} dz_{\hat i}= d(z_iz_{\hat i})=dA_i$.
Since $A_2=A_1+\sum_{j=1}^n (-1)^{j-1}A_j$, we have
\begin{align*}
\iota_{V_1}\cdots \iota_{V_{n-1}}\Omega&=A_ndA_2dA_3\cdots dA_{n-1}+\sum_{j=2}^{n-1}\iota_{V_1}\cdots \iota_{V_{n-2}}z_{\hat j}\iota_{z_{\hat j}} dz_1dz_{\hat 1}\cdots dz_{n-1}dz_{\widehat{n-1}}dA_n \\
&=\sum_{j=2}^{n}(-1)^{n-j} A_jdA_2\cdots \widehat{dA_j}\cdots dA_n\\
&=(-1)^{n-2} A_2dA_3\cdots  dA_n+\sum_{j=3}^{n}(-1)^{n-j} A_j\big(dA_1+(-1)^{j-1}dA_j)dA_3\cdots \widehat{dA_j}\cdots dA_n\\
&=\sum_{k=2}^n(-1)^{n-k} A_kdA_3\cdots  dA_n+\sum_{j=3}^{n}(-1)^{n-j} A_jdA_1dA_3\cdots \widehat{dA_j}\cdots dA_n\\
&= (-1)^{n} A_1dA_3\cdots  dA_n+\sum_{j=3}^{n}(-1)^{n+j} A_jdA_1dA_3\cdots \widehat{dA_j}\cdots dA_n.
\end{align*}
\end{proof}

\subsubsection{Anti-canonical divisor by Schubert divisors}

There is an anti-canonical divisor of $X$ defined as follows, which consists of Schubert divisors of $X$ up to translations by the Weyl group of $SL(n, \mathbb{C})$.
$$D^{\rm Sch}:=\{x_1x_{\hat 1} x_3x_{\hat 3}\cdots x_nx_{\hat n}=0\}\cap X,$$
where we treat the intersection as that for subvarieties in $\mathbb{P}^{n-1}\times \mathbb{P}^{n-1}$. Take the nowhere vanishing holomorphic volume form $\Omega_X^{\rm Sch}$ on $X\setminus D^{\rm Sch}$ given  by
$$\Omega_X^{\rm Sch}:={\Omega\over A_1A_3A_4\cdots A_n}$$
on the open subset $x_1\neq 0, x_{\hat n}\neq 0$. It is meromorhphic on $X$, and has simple poles along the divisor $D^{\rm Sch}$.

\begin{prop}\label{prop:anticanSchu}
We have
$$\iota_{V_1}\cdots \iota_{V_{n-1}}\Omega_X^{\rm Sch}=(-1)^{n-1}d\log f_1\wedge d\log f_2\cdots \wedge d\log f_{n-2},$$
where $f_1={A_1\over A_3}={x_{1}x_{\hat 1}\over x_3x_{\hat 3}}$, and $f_j={A_{j+2}\over A_3}={x_{j+2}x_{\widehat {j+2}}\over x_3x_{\hat 3}}$ for $j=2,\dots, n-2$.
\end{prop}
\begin{proof}
Denote by RHS by the expression  on the right hand side of the expected equality. By direct calculations, we have
\begin{align*}
\text{RHS}&=(-1)^{n-1}({dA_1\over A_1}-{d{A_3}\over A_3})({dA_4\over A_4}-{d{A_3}\over A_3})\cdots ({dA_n\over A_n}-{d{A_3}\over A_3})\\
&=(-1)^{n-1}{A_3dA_1-A_1d{A_3}\over A_1A_3}\cdot {A_3dA_4- A_4d{A_3}\over A_3A_4} \cdots {A_3dA_n- A_nd{A_3}\over A_3A_n}\\
&=(-1)^{n-1}{-A_1d{A_3}dA_4\cdots dA_n  +\sum_{j=3}^n(-1)^{j+1}A_jdA_1dA_3\cdots \widehat{dA_j}\cdots dA_n\over A_1A_3A_4\cdots A_n}\\
&= {\iota_{V_1}\cdots \iota_{V_{n-1}}\Omega\over A_1A_3A_4\cdots A_n}.
\end{align*}
The last equality follows from Lemma \ref{lem:contrusual}.
\end{proof}

Let us consider some other choices of meromorphic volume forms on $X$. Denote
$$B_j:=\sum_{k=j}^{n}(-1)^{k-j}x_kx_{\hat k},\quad j=3,\dots, n. $$
By abuse of notations, we also denote by $B_j$ the function on the open subset $x_1\neq 0, x_{\hat n}\neq 0$. Notice that
$$A_n=B_n,\quad A_j=B_j+B_{j+1},\quad j=3,\dots, n-1.$$
For any $3\leq j\leq n$, we consider a divisor of $X$ defined by
$$D^{\scriptscriptstyle (j)}:=\{x_1x_{\hat 1}x_3x_{\hat 3}\cdots x_{j-1}x_{\widehat{j-1}}B_j\cdots B_n=0\}.$$

Notice that when $j = n$, we have $D^{\scriptscriptstyle (n)} = D^{\rm Sch}$. We also remark that $D^{\rm Rie} := D^{\scriptscriptstyle (3)}$ is the anticanonical divisor which corresponds to Rietsch's mirror \cite{Riet}.

\begin{prop}\label{prop:antigeneral}
For any $3\leq j\leq n$, we have
$$\iota_{V_1}\cdots \iota_{V_{n-1}}{\Omega_X^{\scriptscriptstyle (j)}}=(-1)^{n-j}d\log {x_1x_{\hat 1}\over B_j}\wedge d\log {x_3x_{\hat 3}\over B_j}\wedge \cdots \wedge d\log {x_{j-1}x_{\widehat{j-1}}\over B_j}\wedge  d\log {B_{j+1}\over B_j}\wedge \cdots \wedge d\log{B_n\over B_j},$$
where $$\Omega_X^{\scriptscriptstyle (j)} := {\Omega\over A_1A_3\cdots A_{j-1}B_j\cdots B_n}.$$
\end{prop}
\begin{proof}
Denote by $\text{RHS}^{\scriptscriptstyle (j)}$ the right hand side of the expected equality. Since $D^{\scriptscriptstyle (n)} = D^{\rm Sch}$, by Proposition \ref{prop:anticanSchu} we have
\begin{align*}
\iota_{V_1}\cdots \iota_{V_{n-1}}\Omega_X^{\rm Sch}&=(-1)^{n-1}d\log f_1\wedge d\log f_2\cdots \wedge d\log f_{n-2}\\
&=(-1)^{n-1}(d\log f_1-d\log f_{n-2})\wedge \cdots \wedge (d\log f_{n-3}-d\log f_{n-2})\wedge d\log f_{n-2}\\
&=(-1)^{n-1}d\log {x_1x_{\hat 1}\over x_{n}x_{\hat n}}\wedge d\log {x_4x_{\hat 4}\over x_{n}x_{\hat n}}\wedge \cdots \wedge d\log {x_{n-1}x_{\widehat{n-1}}\over x_{n}x_{\hat n}}\wedge d\log {x_nx_{\hat n}\over x_{3}x_{\hat 3}}\\
&= d\log {x_1x_{\hat 1}\over B_n}\wedge d\log {x_3x_{\hat 3}\over B_n}\wedge \cdots \wedge d\log {x_{n-1}x_{\widehat{n-1}}\over B_n}.
\end{align*}
Namely if $j=n$, then the statement holds.

Now we consider   $3\leq j<n$ and recall $A_j=B_j+B_{j+1}$. Therefore
\begin{align*}
  & \sum_{k=j}^{j+1}(-1)^{k+n}B_kdA_1dA_3\cdots  {dA_{j-1}}dA_jdB_{j+1}\cdots \widehat{dB_k}\cdots dB_n\\
= & (-1)^{j+n}(A_j-B_{j+1})dA_1dA_3\cdots   dA_{j-1}dB_{j+1}dB_{j+2}\cdots dB_n\\
  & \quad+ (-1)^{j+1+n}B_{j+1}dA_1dA_3\cdots  dA_{j-1}dB_jdB_{j+2}\cdots dB_n\\
= & (-1)^{j+n} A_jdA_1dA_3\cdots   dA_{j-1}dB_{j+1}dB_{j+2}\cdots dB_n\\
  & + (-1)^{j+1+n}B_{j+1}dA_1dA_3\cdots  dA_{j-1}dA_jdB_{j+2}\cdots dB_n\\
\end{align*}
Thus by expanding $\text{RHS}^{\scriptscriptstyle (j)}$, we have
\begin{align*}
D^{\scriptscriptstyle (j)}\cdot \text{RHS}^{\scriptscriptstyle (j)}
& = (-1)^{n}A_1d{A_3}\cdots dA_{j-1}dB_{j}\cdots dB_n\\
&   \qquad + \sum_{i=3}^{j-1}(-1)^{i+n}A_idA_1dA_3\cdots \widehat{dA_i}\cdots dA_{j-1}dB_j\cdots dB_n\\
&   \qquad + \sum_{k=j}^{n}(-1)^{k+n}B_kdA_1dA_3\cdots {dA_{j-1}}dB_j\cdots \widehat{dB_k}\cdots dB_n\\
& = (-1)^{n}A_1d{A_3}\cdots dA_{j-1}dA_{j}dB_{j+1}\cdots dB_n\\
&   \qquad + \sum_{i=3}^{j}(-1)^{i+n}A_idA_1dA_3\cdots \widehat{dA_i}\cdots dA_{j-1}dA_jdB_{j+1}\cdots dB_n\\
&   \qquad + \sum_{k=j+1}^{n}(-1)^{k+n}B_kdA_1dA_3\cdots {dA_{j-1}}dA_jdB_{j+1}\cdots \widehat{dB_k}\cdots dB_n\\
& = D^{\scriptscriptstyle (j+1)}\cdot \text{RHS}^{\scriptscriptstyle (j+1)}.
\end{align*}
By induction on $j$, we conclude that, for $3\leq j\leq n$,
$$D^{\scriptscriptstyle (j)}\cdot \text{RHS}^{\scriptscriptstyle (j)}=D^{\scriptscriptstyle (n)}\cdot \text{RHS}^{\scriptscriptstyle (n)}=\iota_{V_1}\cdots \iota_{V_{n-1}}{\Omega}.$$
\end{proof}

To apply Proposition \ref{prop:sLag_fibration}, we need to check that the functions $|f_1|, \dots, |f_{n-2}|$ are Poisson commuting with respect to the Poisson bracket induced by the Fubini-Study K\"ahler form $\omega_{FS}$ on $\mathbb{P}^{n-2}$. To see this, first note that the norms of the inhomogeneous coordinates $w_1 := \frac{y_0}{y_1}, w_2 := \frac{y_2}{y_1}, \dots, w_{n-2} := \frac{y_{n-2}}{y_1}$ are Poisson commuting since they define the log map
$$Log: \mathbb{P}^{n-2}\setminus E \cong (\mathbb{C}^*)^{n-2} \to \mathbb{R}^{n-2}, (w_1,\dots,w_{n-2}) \mapsto (\log|w_1|,\dots,\log|w_{n-2}|),$$
which is a Lagrangian torus fibration; here $E$ is the union of coordinate hyperplanes defined by $y_0y_1\cdots y_{n-2} = 0$. But the functions $f_1, \dots, f_{n-2}$ are pullbacks of the inhomogeneous coordinates $w_1, \dots, w_{n-2}$ by linear isomorphisms $\sigma: \mathbb{P}^{n-2} \to \mathbb{P}^{n-2}$ (given by lower triangular matrices) which obviously preserve $\omega_{FS}$ and hence the Poisson structure. So $|f_1|, \dots, |f_{n-2}|$ are Poisson commuting as well. Proposition \ref{prop:antigeneral} together with Proposition \ref{prop:sLag_fibration} then give the following theorem.

\begin{thm}\label{thm:slag_two-step_flag}
For each $D = D^{\scriptscriptstyle (j)}$, the map
$$\rho := (\mu_1, \ldots, \mu_k, F^*\nu_1, \ldots, F^*\nu_{n-k}): X \setminus D \to \mathbb{R}^n$$
defines a special Lagrangian fibration on $X \setminus D$, where the special condition is with respect to the holomorphic volume form $\Omega_X^{\scriptscriptstyle (j)}$ on $X\setminus D$.
\end{thm}

Observe that the functions $\mu_1, \ldots, \mu_k, F^*\nu_1, \ldots, F^*\nu_{n-k}$ are defined on $X\setminus B$, so we obtain completely integrable systems on these varieties in the sense of \cite[Definition 2.1]{HaKa}. But the fibers over the image of $D \setminus B$ are collapsed tori.

\subsection{The quadrics}\label{sec:Omega_quadric}

In this subsection, we show that the condition \eqref{eqn:condition_holom_vol_form} is satisfied for the pseudotoric structure on the quadric hypersurfaces in $\mathbb{P}^n$ constructed by Tyurin (as in Theorem \ref{thm:pseudotoric_quadric}).

\subsubsection{Even-dimensional quadric}

Here we consider the even-dimensional quadric
$$Q_{2m}=\{x_0x_1+x_2x_3+\cdots+x_{2m}x_{2m+1}=0\}\subset \mathbb{P}^{2m+1}.$$
Recall that the natural action of a maximal compact torus $T^{m+1}\subset T_{\mathbb{C}}^{m+1}$ of $SO(2m+2, \mathbb{C})$ on the quadric was given in \eqref{actionevenqua}.
On the affine plane $x_0\neq 0$, the quadric is defined by  $z_1+z_2z_3+\cdots+z_{2m}z_{2m+1}=0$, where $z_i={x_i\over x_0}$ are the inhomogeneous coordinates.

Let $\theta_k$ denote the weight of $t_k$. Then for $1\leq j\leq m$, the weight of the action of $T^{m+1}$ on the inhomogeneous coordinate $z_{2j}$ (resp. $z_{2j+1}$) is given by $\theta_j-\theta_0$ (resp. $-\theta_j-\theta_0$).
Denote
$$\psi_{m+1}=-\theta_m-\theta_0;\quad \psi_j=\theta_j-\theta_0,\quad j=1, \dots, m.$$
Then the weight on $z_{2j+1}$ is given by $\psi_m+\psi_{m+1}-\psi_j$ for $1\leq j\leq m$.

Let $\{V_1,\dots, V_{m+1}\}$ denote the basis of Hamiltonian vector fields generated by the torus action of $T^{m+1}$ and corresponding to the weights $\psi_j$. Then we have
\begin{equation*}
\left\{
\begin{split}
\iota_{V_j} & = z_{2j}\iota_{z_{2j}}-z_{2j+1}\iota_{z_{2j+1}}, \quad j = 1, \dots, m-1;\\
\iota_{V_m} & = z_{2m}\iota_{z_{2m}}+\sum_{i=1}^{m-1} z_{2i+1}\iota_{z_{2i+1}};\\
\iota_{V_{m+1}} & = z_{2m+1}\iota_{z_{2m+1}} +\sum_{i=1}^{m-1} z_{2i+1}\iota_{z_{2i+1}}.
\end{split}
\right.
\end{equation*}

Similar to the case of two-step flag varieties, we take an anti-canonical divisor $D^{\rm Sch}$ of $X=Q_{2m}$ defined by Schubert divisors, namely, given by
$$D^{\rm Sch} :=\{x_0x_1\cdots x_{2m-3}x_{2m}x_{2m+1}=0\},$$
and consider the meromorphic volume form
$$\Omega_X^{\rm Sch} := {dz_2dz_3\cdots dz_{2m}dz_{2m+1}\over z_1\cdots z_{2m-3}z_{2m} z_{2m+1}}.$$
We denote
$$\Omega := dz_2dz_3\cdots dz_{2m}dz_{2m+1}\text{ and } A_j:=z_{2j}z_{2j+1}\text{ for }j = 1, \dots, m,$$
so that we can also write
$$\Omega_X^{\rm Sch} = {\Omega \over z_1A_1A_2\cdots A_{m-2}z_{2m}z_{2m+1}}.$$

\begin{lemma}\label{lem:contrusualevn}
For $\Omega = dz_2dz_3\cdots dz_{2m}dz_{2m+1}$, we have
$$ \iota_{V_{m+1}}\iota_{V_m}\iota_{V_1}\cdots \iota_{V_{m-1}} \Omega= \sum_{j=1}^{m}(-1)^{m+j} A_jdA_1\cdots \widehat{dA_j}\cdots dA_m.$$
\end{lemma}
\begin{proof}
For $1\leq j\leq m-1$, we have
$$z_{2j+1}\iota_{z_{2j+1}}dA_j=A_j,\quad \iota_{V_j} dz_{2j}d_{z_{2j+1}}=d(z_{2j}z_{2j+1})=dA_j.$$
Moreover, $\iota_{V_{m+1}}\iota_{V_m}=A_{m}\iota_{z_{2m+1}}\iota_{z_{2m}}+\sum_{j=1}^{m-1}z_{2j+1}\iota_{z_{2j+1}} (z_{2m}\iota_{z_{2m}}-z_{2m+1}\iota_{z_{2m+1}})$. Therefore
\begin{align*}
\iota_{V_{m+1}}\iota_{V_m}\iota_{V_1}\cdots \iota_{V_{m-1}}\Omega
& = \iota_{V_{m+1}}\iota_{V_m} dA_1\cdots dA_{m-1} dz_{2m}dz_{2m+1}\\
& = A_mdA_1\cdots dA_{m-1}+\sum_{j=1}^{m-1}(-1)^{m-1}z_j\iota_{z_{2j+1}} dA_1\cdots dA_{m-1}dA_m\\
& = \sum_{j=1}^{m}(-1)^{m+j} A_jdA_1\cdots \widehat{dA_j}\cdots dA_m.
\end{align*}
\end{proof}

\begin{prop}\label{prop:evenSchu}
For $\Omega_X^{\rm Sch} = {dz_2dz_3\cdots dz_{2m}dz_{2m+1}\over z_1z_2\cdots z_{2m-3}z_{2m}z_{2m+1}}$, we have
$$ \iota_{V_{m+1}}\iota_{V_m}\iota_{V_1}\cdots \iota_{V_{m-1}} \Omega_X^{\rm Sch}= d\log f_1\wedge d\log f_2\cdots \wedge d\log f_{m-2}\wedge d\log f_{m},$$
where $f_j = {A_j\over x_0x_1} = {x_{2j}x_{2j+1}\over x_0x_1}$ for $j\in\{1, \dots, m\}\setminus\{m-1\}$.
\end{prop}
\begin{proof}
We notice that $z_1+\sum_{j=1}^mA_j=0$. Similar to the proof of Proposition \ref{prop:anticanSchu}, we calculate the right hand side RHS of the above expression as follows.
\begin{align*}
  & z_1A_1\cdots A_{m-2}A_m\cdot \text{RHS}\\
= & (z_1A_1\cdots A_{m-2}A_m)({dA_1\over A_1}-{d{z_1}\over z_1})\cdots ({dA_{m-2}\over A_{m-2}}-{d{z_1}\over z_1})  ({dA_{m}\over A_m}-{d{z_1}\over z_1})\\
= & {z_1 d{A_1} \cdots dA_{m-2}dA_m + \sum_{1\leq j\leq m\atop j\neq m-1} A_jdA_1d \cdots dA_{j-1}(-dz_1)dA_{j+1}\cdots dA_{m-2}dA_m  }\\
= & {z_1 d{A_1} \cdots dA_{m-2}dA_m + \sum_{1\leq j\leq m\atop j\neq m-1} A_jdA_1d \cdots dA_{j-1}(dA_j+dA_{m-1})dA_{j+1}\cdots dA_{m-2}dA_m  }\\
= & (z_1 +\sum\limits_{1\leq j\leq m\atop j\neq m-1} A_j) d{A_1} \cdots dA_{m-2}dA_m +\sum\limits_{1\leq j\leq m\atop j\neq m-1} (-1)^{m-j}A_jdA_1  \cdots  \widehat{dA_j}\cdots   dA_m  \\
= & z_1A_1\cdots A_{m-2}A_m\cdot \text{LHS}.
\end{align*}
The last equality follows from Lemma \ref{lem:contrusualevn}.
\end{proof}

Following \cite{PRW}, we consider another anti-canonical divisor:
$$D^{\rm Rie} := \{x_0x_1B_1\cdots B_{m-2} x_{2m}x_{2m+1}=0\},$$
where $B_j=\sum_{i=0}^j x_{2i}x_{2i+1}$ for $1\leq j\leq m-2$; more generally, as in the case of the two step-flag variety $F\ell_{1,n-1; n}$, we may consider a sequence of anticanonical divisors of the form
$$D^{\scriptscriptstyle (j)} := \{x_0x_1 A_1\cdots A_{j-1} B_j\cdots B_{m-2}x_{2m}x_{2m+1}=0\},\quad 1\leq j \leq m-1$$
so that $D^{\scriptscriptstyle (m-1)} = D^{\rm Sch}$ and $D^{\scriptscriptstyle (1)} = D^{\rm Rie}$.

Then by similar calculations in the proof of Proposition \ref{prop:antigeneral}, we obtain the following.
\begin{prop}\label{prop:evenPRW}
We have
$$ \iota_{V_{m+1}}\iota_{V_m}\iota_{V_1}\cdots \iota_{V_{m-1}} \Omega_X^{\scriptscriptstyle (j)} = d\log g_1 \cdots \wedge d\log g_{m-2}\wedge d\log g_{m},$$
where
$$ \Omega_X^{\scriptscriptstyle (j)} := {\Omega\over x_0x_1 x_2x_3\cdots x_{2j-2}x_{2j-1}B_j\cdots B_{m-2} x_{2m}x_{2m+1}}$$
and $g_k = {A_k\over x_0x_1}$ for $1\leq k\leq j-1$, $g_k = {B_k\over x_0x_1}$ for $j\leq k\leq m-2$, and $g_m = {x_{2m}x_{2m+1}\over x_0x_1}$.
\end{prop}

By similar reasoning as in the paragraph right before Theorem \ref{thm:slag_two-step_flag}, we can check that the functions $|f_1|,\dots,|f_m|$ as well as $|g_1|,\dots,|g_m|$ are Poisson commuting. So Propositions \ref{prop:evenSchu} and \ref{prop:evenPRW} together with Proposition \ref{prop:sLag_fibration} give the following theorem.

\begin{thm}\label{thm:slag_even-dim_quadric}
For each $D = D^{\scriptscriptstyle (j)}$, the map
$$\rho := (\mu_1, \ldots, \mu_k, F^*\nu_1, \ldots, F^*\nu_{n-k}): X \setminus D \to \mathbb{R}^n$$
defines a special Lagrangian fibration on $X \setminus D$, where the special condition is with respect to the holomorphic volume form $\Omega_X^{\scriptscriptstyle (j)}$ on $X\setminus D$.
\end{thm}


\subsubsection{Odd-dimensional quadric}

Here  we  consider the odd-dimensional quadric
$$Q_{2m-1}=\{x_0^2+x_1x_2+\cdots+x_{2m-1}x_{2m}=0\}\subset \mathbb{P}^{2m}.$$
Recall that the  action of a maximal compact torus $T^{m}\subset T^m_{\mathbb{C}}$ of $SO(2m+1, \mathbb{C})$   on the quadric was given in \eqref{actionoddqua}.
On the affine plane $x_{2m}\neq 0$, the quadric is defined by  $z_0^2+z_1z_2+\cdots+z_{2m-3}z_{2m-2}+z_{2m-1}=0$, where $z_i={x_i\over x_{2m}}$ are the inhomogeneous coordinates.

Let $\theta_k$ denote the weight of $t_k$. Then for $1\leq j\leq m$, the weight  of the action of $T^{m+1}$ on the inhomogeneous coordinate  $z_{2j-1}$ (resp. $z_{2j}$) is given by $\theta_j+\theta_m$ (resp. $-\theta_j+\theta_m$) for $j=1,\dots, m-1$, and the weight on $z_0$ is given by $\theta_m$.

Let $\{V_1,\dots, V_{m}\}$ denote a basis of Hamiltonian vector fields generated by the torus action of $T^{m}$ and corresponding to the weights $\theta_j$. Then we have
$$\iota_{V_m}=\sum\nolimits_{i=0}^{m-1}z_i\iota_{z_{2i}} \quad\text{and}\quad \iota_{V_j}=z_{2j-1}\iota_{z_{2j-1}}-z_{2j}\iota_{z_{2j}},\quad j=1, \dots, m-1.$$

Similar to the case of even-dimensional quadric, we denote
$A_j=z_{2j-1}z_{2j}$ (or $A_j=x_{2j-1}x_{2j}$ when we refer to homogeneous coordinates) for $j=1, \dots, m$ by abuse of notation. Let $B_j=\sum_{i=j}^m A_j$ for $j=2,\dots,m$. We consider the following two anti-canonical divisors $D^{\rm Sch}$, $D^{\rm Rie}$ of $X=Q_{2m-1}$ given respectively by
\begin{align*}
D^{\rm Sch} & := \{x_0A_2\cdots A_m=0\},\\
D^{\rm Rie} & := \{x_0 B_2\cdots B_{m-1} x_{2m-1}x_{2m}=0\}.
\end{align*}
Correspondingly we consider the holomorphic volume forms
\begin{align*}
\Omega_X^{\rm Sch} & := {dz_0dz_1\cdots dz_{2m-2}\over z_0A_2\cdots A_{m-1}z_{2m-1}},\\
\Omega_X^{\rm Rie} & := {dz_0dz_1\cdots dz_{2m-2}\over z_0B_2\cdots B_{m-1}z_{2m-1}}.
\end{align*}
More generally, we may consider a sequence of anticanonical divisors of the form
$$D^{\scriptscriptstyle (j)} := \{x_0A_2\cdots A_{j-1}B_j\cdots B_{m-1}x_{2m-1}x_{2m}=0\},\quad 2\leq j \leq m$$
so that $D^{\scriptscriptstyle (m)} = D^{\rm Sch}$ and $D^{\scriptscriptstyle (2)} = D^{\rm Rie}$, and
$$
\Omega_X^{\scriptscriptstyle (j)} := {dz_0dz_1\cdots dz_{2m-2}\over z_0A_2\cdots A_{j-1} B_j \cdots B_{m-1}z_{2m-1}}.
$$

By similar calculations, we obtain the following.
\begin{prop}\label{prop:odd_quadric}
The following equality holds:
$$
\iota_{V_1} \cdots \iota_{V_{m}} \Omega_X^{\scriptscriptstyle (j)} = \pm d\log {x_0^2\over B_m}\wedge d\log {A_2\over B_m}\wedge \cdots d\log {A_{j-1}\over B_m} \wedge d\log {B_j\over B_m} \wedge d\log {B_{m-1}\over B_m},
$$
and in particular we have
\begin{align*}
\iota_{V_1} \cdots \iota_{V_{m}} \Omega_X^{\rm Sch} & = \pm d\log {x_0^2\over A_m}\wedge d\log {A_2\over A_m}\wedge \cdots \wedge d\log {A_{m-1}\over A_m},\\
\iota_{V_1} \cdots \iota_{V_{m}} \Omega_X^{\rm Rie} & = \pm d\log {x_0^2\over B_m}\wedge d\log {B_2\over A_m}\wedge \cdots \wedge d\log {B_{m-1}\over B_m}.
\end{align*}
\end{prop}

Again, arguing as in the paragraph right before Theorem \ref{thm:slag_two-step_flag}, we see that the two sets of functions given by the norms of those functions obtained in Proposition \ref{prop:odd_quadric} are both Poisson commuting. So
Proposition \ref{prop:odd_quadric} together with Proposition \ref{prop:sLag_fibration} give the following theorem.

\begin{thm}\label{thm:slag_odd-dim_quadric}
For each $D = D^{\scriptscriptstyle (j)}$, the map
$$\rho := (\mu_1, \ldots, \mu_k, F^*\nu_1, \ldots, F^*\nu_{n-k}): X \setminus D \to \mathbb{R}^n$$
defines a special Lagrangian fibration on $X \setminus D$, where the special condition is with respect to the holomorphic volume form $\Omega_X^{\scriptscriptstyle (j)}$ on $X\setminus D$.
\end{thm}

\begin{example} For $m=2$, we have  $\Omega_X^{\rm Sch}=\Omega_X={dz_0dz_1dz_{2}\over z_0z_3}$. Moreover,
\begin{align*}
\iota_{V_2}\iota_{V_1}\Omega_X^{\rm Sch} & = (z_0\iota_{z_0}+z_1\iota_{z_1}+z_{2}\iota_{z_2})(z_1\iota_{z_1}-z_2\iota_{z_2}) {dz_0dz_1dz_{2}\over z_0z_3}\\
& = (z_0\iota_{z_0}+z_1\iota_{z_1}+z_{2}\iota_{z_2}){dz_1z_2dz_0\over z_0z_3}\\
& = {-z_0dz_1z_2+2z_1z_2dz_0\over z_0z_3}\\
& = {z_0dz_3+z_0dz_0^2+2z_1z_2dz_0\over z_0z_3}\\
& = {z_0dz_3-2z_3dz_0\over z_0z_3}\\
& = d\log {z_3\over z_0^2}= - d\log {x_0^2\over x_3x_4}.
 \end{align*}
\end{example}



\section{Towards SYZ mirror symmetry for flag varieties}\label{sec:speculations}


In this section, we speculate on the possible relation between the special Lagrangian fibrations we constructed and Rietsch's LG mirror \cite{Riet} from the perspective of SYZ mirror symmetry \cite{SYZ}.

Let $(X, \omega_X)$ be equipped with a pseudotoric structure $F: X\setminus B \to Y$. Let $D \in |-K_X|$ be an anticanonical divisor satisfying
all the conditions in Proposition \ref{prop:sLag_fibration}, so that
$$\rho = (\mu_1, \ldots, \mu_k, F^*\nu_1, \ldots, F^*\nu_{n-k}): X \setminus D \to \mathbb{R}^n$$
defines a special Lagrangian torus fibration .
We also assume that the map
$$\nu = (\nu_1, \ldots, \nu_{n-k}): Y_0 \to \mathbb{R}^{n-k}$$
is regular, so that all fibers are smooth tori $T$ of real dimension $n-k$ (e.g. when $\nu: Y_0 \to \mathbb{R}^{n-k}$ is given, up to a diffeomorphism of the base, by the toric moment map on $Y$). Then a fiber of $\rho$ is singular if and only if it contains a point where the $T^k$-action is not free. Therefore the discriminant loci of the fibration $\rho: X\setminus D \to \mathbb{R}^n$ is given by the image of the non-free loci of the $T^k$-action on $X$.

In the case of the two-step flag variety $X = F\ell_{1,n-1; n}$, written as a degree $(1,1)$ hypersurface
$$x_{2}x_{\hat 2} = x_1x_{\hat 1} + \sum_{j=3}^n(-1)^{j-1}x_jx_{\hat j}$$
in $\mathbb{P}^{n-1} \times \mathbb{P}^{n-1}$, the non-free loci of the $T^{n-1}$-action is given by the union of the codimension 2 subvarieties:
$$\{ x_j = x_{\hat{j}} = 0 \}.$$
Their images in the base of the fibration $\rho: X \setminus B \to \mathbb{R}^{2n-3}$ give the discriminant loci, which is of real codimension 1. The discriminant loci of $\rho $in the case of the quadrics can be similarly described.

\begin{prop}\label{prop:wall}
For both the quadric hypersurface $Q$ and the two-step flag variety $F\ell_{1,n-1; n}$, a smooth fiber of the special Lagrangian fibration $\rho: X\setminus B \to \mathbb{R}^n$ bounds a nontrivial holomorphic disk if and only if its image under $F$ intersects with the image of the non-free loci of the $T^k$-action.
\end{prop}
\begin{proof}
Let $\varphi: D^2 \to X\setminus D$ be a holomorphic map from the unit disk $D^2 \subset \mathbb{C}$. Recall that the map $F$ maps $X \setminus D$ into $Y \setminus E = Y_0$. So the composition $F \circ \varphi$ gives a map from $D^2$ to $Y_0 \cong (\mathbb{C}^*)^{n-k}$. The maximum principle then implies that $F \circ \varphi$ is a constant map, meaning that the image of $\varphi$ lies in a fiber of $F$. By further composing $\varphi$ with a coordinate function on $X$, we see that one of the coordinates must be zero since otherwise the maximum principle will again be violated. So a smooth fiber of $\rho$ will bound a nontrivial holomorphic disk if and only if it intersects with the union of zero loci of the coordinates on $X$, which is exactly the non-free loci of the $T^k$-action.
\end{proof}


\subsection{The two-step flag variety $F\ell_{1,n-1; n}$}

Let us first look at the case of the two-step flag variety $X = F\ell_{1,n-1; n}$, and describe in more details the wall/chamber structures in the bases of the special Lagrangian torus fibrations.

First of all, although the fibration is not defined on the whole $X$, we can still take the closure of the base of the fibration in $\mathbb{R}^{2n-3}$ to get a convex body. As in the case of $\mathbb{P}^2$ \cite{Aur1} or the Hirzebruch surfaces $\mathbb{F}_2$ and $\mathbb{F}_3$ \cite{Aur2}, this convex body can be obtained by ``pushing in'' some of the facets in the moment polytope (Gelfand-Cetlin polytope) of the central fiber in the toric degeneration of $X$.

We first note that there are exactly 4 facets (codimension 1 faces) in the moment polytope whose preimages in the Gelfand-Cetlin fibration are {\em not} algebraic subvarieties.\footnote{In contrast, if our fibration can be extended to the closure of the base, then the preimage of the boundary should be precisely the anticanonical divisor $D$ we chose.}
All other facets correspond to subvarieties, and they come in pairs -- each pair is the image of a pair of irreducible Schubert divisors of the form
$$x_{j}x_{\hat{j}} = 0, \quad j = 3, \ldots, n.$$
Note that correspondingly there are $2n$ terms in Givental's mirror superpotential.

Going from the Gelfand-Cetlin fibration to our fibration with $D = D^{\scriptscriptstyle (n)}$, a conifold singularity is being smoothed out, so the 4 facets will become a union of two sets each consisting of 2 facets, which are the images of the 2 sets:
$$\{x_1x_{\hat{1}} = 0\}, \quad \{x_2x_{\hat{2}} = 0\},$$
and one of them, namely the one corresponding to $\{x_2x_{\hat{2}} = 0\}$, will be ``pushed in'' to form a wall in the base of the fibration (see e.g. the picture in \cite{CPU}). Note that the non-free loci inside $X \setminus D$ is easy to describe in this case, namely, simply given by the codimension 2 subvariety
$$x_{2}x_{\hat{2}} = 0.$$
By Proposition \ref{prop:wall}, Lagrangian torus fibers which bound nontrivial holomorphic disks in $X \setminus D$ are those which intersect nontrivially with the divisors
$$\{x_2 = 0\}, \quad \{x_{\hat{2}} = 0\},$$
whose images under $\rho$ produce a single wall in the base.
In terms of the mirror superpotentials, this means 4 of the terms in Givental's mirror now combine to 2 terms in the new mirror (more precisely, this means the mirror superpotential over one of the chambers in the base).
So the superpotential mirror to $(X, D^{\text{Sch}})$ should have
$$ 2 + (2n - 4) = 2n - 2 $$
terms.

Then, for each $j > 3$, going from our fibration with $D = D^{\scriptscriptstyle (j)}$ to that with $D = D^{\scriptscriptstyle (j-1)}$, we are pushing in the pair of facets corresponding to
$$x_{j-1}x_{\widehat{j-1}} = 0,$$
and the facets become a wall inside the interior of the convex body. In terms of the mirror superpotentials, this means that in each step, 2 terms in the previous mirror superpotential combine into one single term in the new mirror superpotential.

So at the end of this whole process (i.e. when $j = 3$), we arrive at the anticanonical divisor $D^{\text{Rie}}$ chosen by Rietsch.
There are $n - 3$ walls in the base of the fibration, which are images of the following $n - 2$ pairs of Schubert divisors:
$$\{x_2x_{\hat{2}} = 0\},\quad \{x_3x_{\hat{3}} = 0\},\quad \ldots,\quad \{x_{n-1}x_{\widehat{{n-1}}} = 0\}.$$
The number of facets of the convex body is now given by
$$ 2 + (n - 3) + 2 = n + 1, $$
which is exactly the number of terms in Rietsch's LG mirror.

\quad\\
\begin{center}
\begin{tabular}{ | l | l | l | }
\hline
Givental & Schubert ($D^{\rm Sch} = D^{\scriptscriptstyle (n)}$) & Rietsch ($D^{\rm Rie} = D^{\scriptscriptstyle (3)}$) \\
\hline
$2n = 4 + 2(n-3) + 2$ & $2n - 2 = 2 + 2(n-3) + 2$ & $n + 1 = 2 + (n-3) + 2$ \\
\hline
\end{tabular}
\end{center}
\quad\\

\begin{remark}
  By Givental's mirror superpotential for $F\ell_{1, n-1;n}$, we mean the one given by Batyrev,  Ciocan-Fontanine,   Kim and   Straten \cite{BCFKS} for a general partial flag manifold  $F\ell_{n_1, \dots, n_k; n}$ which generalizes that of a complete flag manifold $F\ell_{1, 2,\dots, n-1;n}$ constructed by Givental \cite{Givental}. The same superpotential was recovered by Nishinou,  Nohara and  Ueda  \cite{NNU} by a different approach using the Gelfand-Cetlin toric degeneration and counting of holomorphic disks.

  For complex Grassmannians, Marsh and Rietsch constructed a mirror superpotential in terms of the Pl\"ucker coordinates in \cite{MaRi}, which is isomorphic to the Lie-theoretic mirror superpotential for a general homogeneous variety $G/P$ constructed earlier by Rietsch \cite{Riet}.
  There should be a construction of the mirror superpotential for $F\ell_{1, n-1; n}$ in terms of the Pl\"ucker coordinates and similar to that in \cite{MaRi}, which we refer to as Rietsch's mirror in the above table. However, the precise expression of such a superpotential, though should be known to experts, is still missing in the literature.

  The superpotential with respect to the choice $D^{\rm Sch}$ is expected to exist, but again its precise expression is unknown.
\end{remark}

\begin{remark}\label{rmkRietmirror}
  For $n=3, 4$, we can  interpret Rietsch's  mirror   \cite{Riet}  for $F\ell_{1,n-1; n}$ directly by the definition therein in terms of the Pl\"ucker coordinates as follows.  \begin{align*}
   F\ell_{1,2; 3}: & \quad D^{\rm Rie}=\{x_1x_{23}x_3x_{12}=0\},\qquad W^{\rm Rie}= {x_2\over x_1}+{x_{13}\over x_{12}}+q_1{x_{13}\over x_{23}}+q_{2}{x_2\over x_3};\\
      F\ell_{1,3; 4}: &\quad D^{\rm Rie}=\{x_1x_{234}(x_3x_{124}-x_4x_{123})x_4x_{123}=0\},\\
         &\quad W^{\rm Rie}= {x_3\over x_1}+ {x_{134}\over x_{123}}+q_1{x_{134}\over x_{234}}+ q_2{x_3\over x_4}+   {x_2x_{124}\over x_3x_{124}-x_{4}x_{123}}.
  \end{align*}
\end{remark}
\subsection{The quadrics}

For an $N$-dimensional quadric $Q_N \subset \mathbb{P}^{N+1}$, the Gelfand-Cetlin polytope has $N + 2$ facets (see e.g. \cite[Section 3]{NNU}), so the number of terms in Givental's mirror superpotential is $N + 2$. The closure of the base of our fibration is again a convex body obtained by ``pushing in'' some of the facets in the Gelfand-Cetlin polytope  of the central fiber in a toric degeneration of $Q_N$.

Like the two-step flag variety $F\ell_{1,n-1; n}$, there are exactly 4 facets (codimension 1 faces) in the Gelfand-Cetlin polytope whose preimages in the Gelfand-Cetlin fibration are {\em not} algebraic subvarieties.
Going from the Gelfand-Cetlin fibration to our fibration for $D^{\rm Sch}$, a conifold singularity is being smoothed out and these 4 facets become a union of two sets
in which one of them is ``pushed in'' to form a wall in the base of the fibration.

All other facets correspond to subvarieties, and $2m - 4$ of them comprise $m - 2$ pairs (where $m = N/2$ when $N$ is even and $m = (N + 1)/2$ when $N$ is odd). When we go from our fibration with $D = D^{\scriptscriptstyle (j)}$ to that with $D = D^{\scriptscriptstyle (j-1)}$, successive pairs of facets are being pushed in to form walls in the interior and accordingly pairs of terms are being combined into single terms in the new mirror superpotential. At the end, we arrive at the anticanonical divisor $D^{\rm Rie}$ which corresponds to Rietsch's construction
and the number of facets of the convex body is given by
$$\left\{
\begin{array}{ll}
m + 2 & \text{when $N = 2m$ is even};\\
m + 1 & \text{when $N = 2m - 1$ is odd},
\end{array}
\right.$$
which is exactly the number of terms in Rietsch's LG mirror.

\quad\\
For $N = 2m$:
\begin{center}
\begin{tabular}{ | l | l | l | }
\hline
Givental & Schubert ($D^{\rm Sch} = D^{\scriptscriptstyle (m-1)}$) & Rietsch ($D^{\rm Rie} = D^{\scriptscriptstyle (1)}$) \\
\hline
$2m + 2 = 4 + 2(m-2) + 2$ & $2m = 2 + 2(m-2) + 2$ & $m + 2 = 2 + (m-2) + 2$ \\
\hline
\end{tabular}
\end{center}

\quad\\
For $N = 2m - 1$:
\begin{center}
\begin{tabular}{ | l | l | l | }
\hline
Givental & Schubert ($D^{\rm Sch} = D^{\scriptscriptstyle (m)}$) & Rietsch ($D^{\rm Rie} = D^{\scriptscriptstyle (2)}$) \\
\hline
$2m + 1 = 4 + 2(m-2) + 1$ & $2m - 1 = 2 + 2(m-2) + 1$ & $m + 1 = 2 + (m-2) + 1$ \\
\hline
\end{tabular}
\end{center}
\quad\\

\begin{remark}\label{rem:Giv_Rie}
	By Givental's mirror of a quadric here, we mean the LG superpotential obtained from the toric degeneration described in \cite[Section 3]{NNU}, which is analogous to the construction of Givental's mirror for flag manifolds of type A \cite{Givental, BCFKS, NNU00}.
	
	On the other hand, by Rietsch's mirror, we mean the LG superpotential given in \cite{PeRi, PRW}, which is isomorphic to the original Lie-theoretic construction of Rietsch \cite{Riet}.
	
	The mirror superpotential with respect to  $D^{\rm Sch}$ is expected to exist, but its precise expression is unknown.
\end{remark}

\end{document}